\documentclass[12pt]{amsart}

\usepackage{amsmath,amssymb}
\usepackage{color}

\usepackage{graphicx}

\newcommand{\R}{\mathbb R}
\newcommand{\s}{\mathbb S}

\newcommand{\h}{\mathcal H}

\newcommand{\CS}{\mathrm{C}^{\infty}(\mathbb{S}^{1})}
\newcommand{\hCS}{\mathrm{\hat C}^{\infty}(\mathbb{S}^{1})}
\newcommand{\nCS}{\mathrm{C}_0^{\infty}(\mathbb{S}^{1})}
\newcommand{\Circle}{\mathbb{S}^{1}}
\newcommand{\VectS}{\mathrm{Vect}^{\infty}(\mathbb{S}^{1})}
\newcommand{\Vect}{\mathrm{Vect}}

\newtheorem{thm}{Theorem}[section]
\newtheorem{cor}[thm]{Corollary}

\newtheorem{prop}[thm]{Proposition}

\theoremstyle{definition}

\theoremstyle{remark}

\begin{document}

\title[Restrictions on the geometry]{Restrictions on the geometry of the \\
periodic vorticity equation}

\author{Joachim Escher}
\address{Institute for Applied Mathematics, University of Hanover, D-30167 Hanover, Germany}
\email{escher@ifam.uni-hannover.de}
\author{Marcus Wunsch}
\address{RIMS, Kyoto University, Kyoto 606-8502, Sakyo-ku Kitashirakawa Oiwakecho,  Japan}
\email{mwunsch@kurims.kyoto-u.ac.jp}



\maketitle

\begin{abstract}
We prove that several evolution equations arising as mathematical models for fluid motion cannot be realized as metric Euler equations on the Lie group {\sc Diff}$^\infty(\s^1)$ of all smooth and orientation-preserving diffeomorphisms on the circle. These include the quasi-geostrophic model equation, cf. \cite{ccf-ann}, the axisymmetric Euler flow in $\R^d$ \cite{oz}, and De Gregorio's vorticity model equation as introduced in \cite{deg1}.
\end{abstract}

\section{Introduction}
\noindent
In this article, we are concerned with certain geometric aspects of members of the family of fractional differential equations given by
\begin{eqnarray} \label{proto}
\begin{cases}
m_t + u m_x + b \; u_x m = 0, \quad t > 0,\ x\in\s^1 \\[0.1cm]
m = (-\Delta)^{a/2}\;u + \int_{\s^1} u(x)\;dx,\quad u\in \CS, 
\end{cases}
\end{eqnarray}
where $a$, $b$ are real numbers, and where $\s^1 = \R / \mathbb Z$ denotes the unit circle of length $1$. If $a\ne0$, the operator 
$(-\Delta)^{a/2}$ has to be understood as the Fourier multiplication operator on $\CS$ induced by the symbol $\vert k\vert^a$ with $k\ne 0$. The case $a=0$ is particular. Here we simply set $m=u$ in (\ref{proto}).
 \par
The family \eqref{proto} is a prototype for many evolution equations pertaining to the mathematical modeling of fluid dynamics. 
\\
For solutions $u$ with vanishing spatial mean (a property which is preserved by the flow of \eqref{proto}), we distinguish the following important cases. 
In the case of $a = 0$, \eqref{proto} reduces to the well-known Burgers equation, cf. \cite{burgers}. 
On the other hand, if $a = 1$, then one obtains De Gregorio's vorticity model equation \cite{deg1} for $b = -1$, and the quasi-geostrophic model equation of \cite{ccf-ann, ccf} (cf. also \cite{CCCF}) in one space dimension (coinciding with the Birkhoff-Rott model equation) for $b = 1$. 
For arbitrary $b \neq 0$, we get the generalized CLM (Constantin-Lax-Majda) equation \cite{osw,clm}, a one-dimensional model for the three-dimensional vorticity equation, with parameter $\alpha = -1/b$.  
Finally, if $a = 2$, then \eqref{proto} simplifies to the Hunter-Saxton equation, cf. \cite{hu:sa, len:weak, len08}, an equation modeling the propagation of weakly nonlinear orientation waves in a massive nematic liquid crystal if $b = 2$, and if $b = 3$, the Burgers equation reappears in disguise, differentiated twice in space. 
For integers $d$, the prescription $b = (d - 3)/(d - 1)$ turns \eqref{proto} into the equation describing the axisymmetric Euler flow in $\R^d$, see \cite{pj,oz,cw,st}, while arbitrary values of $b$ correspond to the generalized Proudman-Johnson equation with parameter $\alpha = -b$, cf. \cite{okamoto,mw-pj, chow}. \\
For $a = 2$, the authors of \cite{lmt} studied solutions to \eqref{proto} whose mean does not vanish. 
Special cases include the $\mu$HS equation \cite{klm} (b = 2) lying ''mid-way'' between the Camassa-Holm equation \cite{ch} describing the unidirectional irrotational free surface flow of a shallow layer of an inviscid fluid,  and the Hunter-Saxton equation, and also, for $b = 3$, the $\mu$DP equation\footnote{
{\sc J. ~Escher, M. ~Kohlmann, B. ~Kolev}, \\
Geometric aspects of the periodic $\mu$DP equation (preprint arXiv:1004.0978)}, which is a generalization of the Degasperis-Procesi equation \cite{dp,ek}. \par
Some of the special cases of \eqref{proto} are not only relevant 
in hydrodynamics, they also play an important geometric role as Euler equations for the geodesic flow on the group of orientation-preserving diffeomorphisms of the circle $\s^1$ (modulo the subgroup of rigid rotations if the means of the solutions vanish) with respect to a Riemannian metric. 
These particular cases are, for zero-mean solutions, the Burgers equation ($a = 0$, $b = 2$), the generalized CLM equation ($a = 1$, $b = 2$), and the Hunter-Saxton equation ($a = b = 2$). 
There also exists a Riemannian connection in the case of the $\mu$HS equation \cite{klm}. The geodesic flow induced by the generalized CLM equation has also been studied in some detail\footnote{{\sc J. ~Escher, B. ~Kolev, M. ~Wunsch}, \\ The geometry of a vorticity model equation (preprint)}.
\par
The existence of a {\it Riemannian} connection, as shown in the groundbreaking study of \cite{ek}, is, however, {\em{not}} necessary for recasting evolution equations such as \eqref{proto} as Euler equations for geodesic flows: One can also define geodesic flows with respect to {\it linear} connections. 
The corresponding concept of non-metric Euler equations (to be explained in Section 2) allows us to interpret \textit{any} of the members of the family \eqref{proto} as geodesic flows with respect to a linear connection. 
\par
Nevertheless, the metric case is of vital importance, since qualitative properties of solutions may be rooted in the geometrical structure of \eqref{proto} (cf. \cite{kolev} for a discussion of the case for the $b$-equation, in which $a=2$). 
Therefore 
the object of this study lies in finding values of $b$ for which the linear connection associated to \eqref{proto} \textit{does not} coincide with a Riemannian one. Our result is complete in the case $a\ne 1$; in the case $a=1$, we get a quite satisfactory result, which allows also for interesting applications in this case. 
In particular, we will find that the quasi-geostrophic model equation \cite{ccf-ann}, the axisymmetric Euler flow in $\R^d$ \cite{oz}, and De Gregorio's vorticity model equation \cite{deg1} can only be realized as non-metric Euler equations. 
\par
In Table \ref{paradigm}, we summarize three paradigmatic examples. It shows special cases of (\ref{proto}) for the three inertia operators $\hbox{id}_{\s^1}$,  $\Lambda_\mu$ and $\Lambda^2_\mu$, respectively. Note that the corresponding metrics then are induced by $L_2(\s^1)$, $H^1(\s^1)$, and $H^2(\s^1)$, respectively, in the general case of non-zero mean 
solutions, and by the homogeneous Sobolev spaces $\dot H^1(\s^1)$ and $\dot H^2(\s^1)$ in the case of zero mean solutions.

\begin{table} \label{paradigm}
\begin{center}
\caption{The family of equations (\ref{proto}) induced by the inertia operators $\hbox{id}_{\s^1}$, $\Lambda_\mu$, and $\Lambda_\mu^2$. Italics indicate equations satisfied by evolutions with vanishing spatial mean, while bold letters highlight equations with non-zero mean solutions. We single out the second row, which contains the equations for which \eqref{proto} can be realized as a metric Euler equation, while the section below incorporates non-metric Euler equations. This will be seen from the analysis in Section \ref{metricity}. }
\end{center}
\begin{center}
          \begin{tabular}{ | p{2.3cm}   | p{1.55cm} | p{3.3cm} | p{3.9cm} |}
    \hline 
         &  \vfill $\hbox{id}_{\s^1}$ &  \vfill $\Lambda_\mu = \sqrt{-\Delta} + \mu(\cdot)$ & \vfill $\Lambda^2_\mu = -\Delta + \mu(\cdot)$ \\[0.1cm]  \hline \hline 
  \hfill \vfill $b = 2$ &\vfill \textit{Burgers equation} & \vfill \textit{metric} {\it gCLM equation} \cite{mw-geod}& \vfill{\it Hunter-Saxton equation} \cite{hu:sa},
    \\ & \cite{burgers}&  & {\bf $\mu$HS equation} \cite{klm}\\[0.1cm] \hline \hline
    \vfill $b = 3$ &  & & \vfill {\bf $\mu$DP equation} \cite{lmt}\\[0.1cm]  \hline
    \vfill $b = 1$ & &{\it quasi-geostrophic model equation} \cite{ccf-ann}&  \\[0.1cm]  \hline
    \vfill $b = -1$ & & {\it vorticity model equation} \cite{deg1} & \\[0.1cm]
    \hline 
   \vfill $b = \frac{d - 3}{d - 1}$ & & & {\it axisymmetric Euler flow in} $\R^d$, $d \ge 2$ \cite{oz}\\[0.1cm] \hline
    \vfill $b = -a = \frac{-1}{\alpha}$ & & \vfill {\it gCLM vorticity model equation} & \vfill{\it gPJ equation with parameter $a$} \cite{okamoto},\\
    &&{\it with parameter $\alpha$} \cite{osw}& $\mu${\bf{b-equation}} \cite{lmt} \\ \hline
    
    \end{tabular}
\end{center}
\end{table}

\section{Euler Equations on Lie Groups}
\noindent
{\sc V. Arnold} pointed out in his seminal paper \cite{ar} that the Euler equations describing the motion of a perfect fluid can be recast as the geodesic flow for right-invariant metrics on the Lie group of volume-preserving diffeomorphisms. Subsequently, {\sc Ebin \& Marsden} rigorously justified this geometric picture in \cite{em}. 
While this general Euler equation was at first derived for the Levi-Civita connection of a one-sided invariant Riemannian metric on a Lie group $G$, {\sc Escher \& Kolev} found that the theory can be extended to the more general setting of a one-sided invariant linear connection \cite{ek}. In what follows, we shall give a short account of this generalization for the readers' convenience. \par
Let $G$ be a Lie group and $\frak{g}$ its Lie algebra (the tangent space of $G$ at its unit element $e$). An isomorphism $A : \frak{g} \rightarrow \frak{g}^\ast$ which is symmetric with respect to the inner product on $\frak{g}$, 
$$\langle Au, v \rangle = \langle Av, u \rangle, \quad \mbox{ for all } u,\;v \in \frak{g},$$
is called an {\it inertia operator} on $G$. By right translation, $A$ gives rise to a right-invariant metric on $G$ which we shall henceforth denote by $\varrho_A$. \\
Let $[\cdot,\cdot]$ be the Lie bracket on the smooth sections of the tangent bundle over $G$, and define the bilinear operator $B$ by
$$B(u,v) = \frac{1}{2} \left[ \left( \mbox{ad}_u\right)^\ast (v) + \left( \mbox{ad}_v\right)^\ast (u) \right],$$
where $(\mbox{ad}_u)^\ast$ is the adjoint with respect to the induced metric $\varrho_A$ of the natural action of $\frak{g}$ on itself. In analogy with the Christoffel symbols in finite-dimensional Riemannian geometry, $B$ is called {\it Christoffel operator}. It then turns out that the Levi-Civita connection $\nabla$ on $G$ induced by $\varrho_A$ can be represented in terms of the Lie bracket $[\cdot,\cdot]$ and the Christoffel operator $B$ as 
\begin{equation} \label{levi}
\nabla_{\xi_u} \; \xi_v = \frac{1}{2} [\xi_u,\xi_v] + B(\xi_u, \xi_v),
\end{equation}
where $\xi_u$ is the right-invariant vector field on $G$ 
generated by $u\in\frak{g}$.
These statements, as well as the proposition below, were proven in \cite{ek}.
\begin{prop}\label{gEuler}
A smooth curve $g(t)$ on a Lie group $G$ is a geodesic for a right-invariant linear connection $\nabla$ defined by \eqref{levi} if and only if its Eulerian velocity $u = g' \circ g^{-1}$ satisfies the first-order equation
\begin{equation} \label{gen:eul}
u_t = - B(u,u).
\end{equation}
This equation is known as the {\bf Euler equation} on $G$ with respect to $A$.
\end{prop}

Note that (\ref{levi}) defines a right invariant linear connection for any bilinear operator $B\,:\,\frak{g}\times\frak{g}\to\frak{g}$. In general, however, this connection is not compatible with a Riemannian structure. If there is no Riemannian metric on $G$ which is preserved by the connection \eqref{levi}, we shall call \eqref{gen:eul} a {\it non-metric Euler equation} on $G$.  

\section{Metricity} \label{metricity}
\noindent
In this section, we shall prove the main result of our paper. We first specialize the setting of the preceding section to the case 
$G=${\sc Diff}$^\infty(\s^1)$,  the Fr\'{e}chet Lie group of all smooth and orientation preserving diffeomorphisms on the circle $\s^1$.
\par

Since the tangent bundle $T\Circle \simeq \s^1 \times\, \R$ is trivial, $\Vect^\infty(\s^1)$, the space of smooth vector fields on $\Circle$, 
can be identified with $\CS$.  
The Lie bracket on $\VectS \simeq \CS$ is given by\begin{footnote}{
Notice that this bracket differs from the usual bracket of vector fields by a sign.}
\end{footnote}
\begin{equation*}
    [u,v] = u_{x}v - uv_{x}.
\end{equation*}
The topological dual space of $\VectS\simeq \CS$ is given by the distributions $\Vect'(\mathbb{S}^1)$ on $\mathbb{S}^1$. 
In order to get a convenient representation of the Christoffel operator $B$, we restrict ourselves to $\Vect^\ast(\mathbb{S}^1)$, 
the set of all regular distributions, which may be represented by smooth densities, i.e.,\ $S\in\Vect^\ast(\mathbb{S}^1)$ iff
there is an $m\in\CS$ such that 
\begin{equation*}
 \langle S, u\rangle=\int_{\mathbb{S}^1}m u\,dx\quad\hbox{for all}\quad  u\in\CS.
\end{equation*} 
 
By Riesz' representation theorem, we may identify the vector spaces
$\Vect^\ast(\mathbb{S}^1)\simeq \CS$. In the following, we denote by $\mathcal{L}_{is}^{sym}(\CS)$ the set of all continuous isomorphisms 
on $\CS$ which are symmetric with respect to the $L_2(\mathbb{S}^1)$ inner product.
Each $A\in\mathcal{L}_{is}^{sym}(\CS)$ is called a {\em{regular inertia operator}} on {\sc Diff}$^\infty(\s^1)$.
\par
As shown in \cite{ek}, given any regular inertia operator $A\in\mathcal{L}_{is}^{sym}(\CS)$, the Christoffel operator is given by
\begin{equation}\label{chris}
 B(u,v)=\frac{1}{2}A^{-1}[2Au\cdot v_x + 2Av\cdot u_x+ u\cdot (Av)_x+ v\cdot(Au)_x]
\end{equation}
for all $u,\;v\in \CS$.
\\
\par

We now introduce a particular class of regular inertia operators on  {\sc Diff}$^\infty(\s^1)$.
Given $a\in\R\setminus\{0\}$, we set
\begin{equation*}
\Lambda_\mu^a u:=  (-\Delta)^{a/2}\;u + \mu(u),\quad u\in \CS,
\end{equation*}
where $(-\Delta)^{a/2}$ stands for the Fourier multiplication operator with symbol $\vert k\vert ^a$ for $k\ne 0$,  and where $\mu$ is the projection $\mu(u):=\int_{\s^1}u(x) \,dx$. The case $a=0$ is particular. Here we set $\Lambda^0_\mu:=\hbox{id}_{\s^1}$.
Observe that $\Lambda_{\mu}^a\in \mathcal L_{is}^{sym}(C^\infty(\s^1))$ for all $a\in\R$ and that $\Lambda_{\mu}^a{\bf 1}={\bf 1}$,
where ${\bf 1}$ is the constant function assigning the value $1$ to all $x \in \s^1$.

\begin{prop}\label{MR}
Given any numbers $a, b \in \R$, consider the doubly parameterized family of fractional differential equations
\begin{eqnarray} \label{proto1}
\begin{cases}
m_t + u m_x + b \; u_x m = 0, \quad t > 0,\ x\in\s^1,\\[0.1cm]
m = \Lambda^a_{\mu} u 
\end{cases}
\end{eqnarray}
and  assume that there is an
inertia operator ${A_{a,b}}\in \mathcal L^{sym}_{is}(\CS)$ such that
(\ref{proto1}) 
is the Euler equation on {\sc Diff}$^\infty(\s^1)$ with respect to the metric $\rho_{A_{a,b}}$ induced by $A_{a,b}$.

If either $a \in \R\setminus \{1\}$ or $a = 1$ and $b \in [-1,\infty)$, then $b = 2$ and $A_{a,2} = \Lambda^a_ \mu$. 
\end{prop}

\begin{proof}
Let $a,\,b \in \R$ be given and assume that \eqref{proto} is the Euler equation on {\sc Diff}$^\infty(\s^1)$ 
with respect to $\rho_{A_{a,b}}$.  If no confusion seems likely, we simply write $A$ for ${A_{a,b}}$. From (\ref{proto1}), 
Proposition \ref{gEuler}, and (\ref{chris}) we conclude that
\begin{equation} \label{funda}
A^{-1} [2 Au \;u' + u \;(Au)'] = (\Lambda^{a}_\mu)^{-1} [b \Lambda^a_\mu u \;u' + u \;(\Lambda^a_\mu u)']
\end{equation}
for all $u \in \CS$. \par
(a) Inserting $u = {\bf 1}$ into (\ref{funda}), we see that $A{\bf 1}$ is constant,
which we normalize to ${\bf 1}$.
Next, choosing $u + \lambda {\bf 1}$, with $u \in \CS$ and $\lambda \in\R$, the left-hand side of \eqref{funda} becomes
\begin{eqnarray*}
&&\frac{1}{\lambda} A^{-1} [2 (A u + \lambda) \;u' + (u + \lambda) \; (A u)' ] \\
&&= A^{-1} \left[ \frac{2 Au \;u' + u \;(Au)'}{\lambda} + 2u' + (Au)'\right].
\end{eqnarray*}
Letting $ \lambda \rightarrow \infty$ in the latter expression, we get
\begin{equation*} A^{-1} [2 u' + (Au)'].
\end{equation*}
The same substitution on the right-hand side gives
\begin{eqnarray*}
\frac{1}{\lambda} (\Lambda^{a}_\mu)^{-1} [b (\Lambda_\mu^a u + \lambda) \;u' + (u + \lambda) \; (\Lambda_\mu^a u)' ],
\end{eqnarray*}
which leads, for $\lambda \rightarrow \infty$, to the expression
\begin{equation*} 
(\Lambda_\mu^{-a})^{-1} [b u' + (\Lambda_\mu^a u)']. 
\end{equation*}
\noindent
Upon setting $\mathbf{e}_k (x):= e^{ikx} \in \CS$ with $k\ne 0$, identification of these limits shows that 
$$
2ik \mathbf{e}_k + (A\mathbf{e}_k)' =  \frac{1}{|k|^a} A \left[ b ik \mathbf{e}_k + ik |k|^a \mathbf{e}_k \right],
$$
so that the ordinary differential equation for $v_k = A\mathbf{e}_k$ reads
\begin{equation}\label{ODEv}
v_k' - i r_k v_k = - 2 ik \mathbf{e}_k,
\end{equation}
where $r_k = \frac{k}{|k|^a} [b + |k|^a]$. By solving (\ref{ODEv}) explicitly, one 
sees that $b \neq 0$; since otherwise, there is no periodic solution to (\ref{ODEv}). If $b\ne 0$, we have
$$
v_k (x) = \gamma_k e^{i r_kx} + \beta_k \mathbf{e}_k(x),
$$
where 
\begin{equation} \label{beta}
\beta_k = \frac{2|k|^a}{b}\quad \hbox{for}\quad k \in \mathbb Z^\ast:=\mathbb Z\setminus\{0\}.
\end{equation}
\par
(b) We first assume that $\gamma_k = 0$ for all integers $k$ in $\mathbb Z^\ast$. 
In this case, choosing $u = \mathbf{e}_k$ in \eqref{funda} yields
\begin{eqnarray*}
|2k|^a \left[2 \beta_k \mathbf{e}_k \;ik \mathbf{e}_k + \mathbf{e}_k \;\beta_k ik \mathbf{e}_k \right] = \beta_{2k} \left[ b |k|^a \mathbf{e}_k \; ik \mathbf{e}_k + \mathbf{e}_k |k|^a ik \mathbf{e}_k\right]
\end{eqnarray*}
and thus 
$$
3 |2k|^a \; \beta_k \mathbf{e}_{2k} = (1 + b) |k|^a \; \beta_{2k} \mathbf{e}_{2k}. 
$$
Comparison of the coefficients, together with \eqref{beta}, forces $b = 2$ and $A = \Lambda_\mu^a$. \par
(c) Assume now that there exists a $p \in \mathbb Z^\ast$ such that $\gamma_p \neq 0$. In the following, we shall see that this assumption leads to a contradiction if $a \in \R\setminus\{1\}$ or if $a = 1$ and $b \in [-1,\infty)$.

First, periodicity ensures that $r_p =: m$ is an integer. 
Thus $b =  \frac{k}{p} |p|^a$ for another integer $k := m - p$. 
Observe also that $1 \neq m \neq p$ since $b \neq 0$.
Furthermore, the symmetry of $A$ implies that $r_m=p$. Indeed, we have 
$$
(A\mathbf{e}_p\vert \mathbf{e}_m)_{L_2}=(\gamma_p e^{imx}\vert \mathbf{e}_m)_{L_2}=\gamma_p
$$
and therefore 
$$
\gamma_p=(\mathbf{e}_p\vert A\mathbf{e}_m)_{L_2}={\overline{\gamma}}_m (\mathbf{e}_p\vert e^{i{r_m}x})_{L_2}.
$$
Consequently, $\gamma_p\ne 0$ forces $(\mathbf{e}_p\vert e^{i{r_m}x})_{L_2}\ne 0$, which implies $p=r_m$.
Moreover, $m = r_p = k + p \neq 0$. 
We may thus calculate
\begin{eqnarray*}
p &=& r_m = r_{k + p} = \frac{k + p}{|k + p|^a} \; [b + |k + p|^a] \\
&=& \frac{k + p}{|k + p|^a} \; b + k + p,
\end{eqnarray*}
so that 
\begin{equation*} \label{fallu}
- |k + p|^a = (k + p) \frac{b}{k} = (k + p) \frac{1}{p} |p|^{a}.
\end{equation*}
This identity clearly implies that
\begin{equation} \label{sign}
- \hbox{sign}(k + p) =\hbox{sign}\,p
\end{equation}
as well as
\begin{equation} \label{kp}
 \left|\frac{k + p}{p}\right|^{a-1} = 1.
\end{equation}

Let us first inspect the case $a\in\R\setminus\{1\}$. Then (\ref{kp}) implies that $\vert k+p\vert=\vert p\vert$ and we conclude from (\ref{sign})
that $k=-2p$, which in turn yields
$b=-2\vert p\vert^a$. We also conclude that $\gamma_n=0$, whenever $\vert n\vert\ne\vert p\vert$, since otherwise the same arguments as above show that $b=-2\vert n\vert^a$, which is impossible by our assumption $\vert n\vert\ne\vert p\vert$.

Next, we wish to insert $u=\mathbf{e}_p$ into (\ref{funda}). In order to do so, we remark that 
\begin{equation} \label{values}
 r_p=k+p = -p\quad\hbox{and}\quad\beta_p=\frac{2\vert p\vert^a}{b}=-1,
\end{equation}
since $k=-2p$ and $b=k\vert p\vert ^a/p$. Using (\ref{values}), we find
\begin{equation*}
2 (A\mathbf{e}_p)\mathbf{e}'_p+\mathbf{e}_p(A\mathbf{e}_p)'=ip\gamma_p\mathbf{1}-3ip\mathbf{e}_{2p}.
\end{equation*}
But $2p$ is different from $p$ and $-p$, thus $\gamma_{2p}=0$ and so $A^{-1}\mathbf{e}_{2p}=\mathbf{e}_{2p}/\beta_{2p}$.
Hence the left-hand side of (\ref{funda}) with $u=\mathbf{e}_p$ is
\begin{equation}\label{rh}
ip\gamma_p\mathbf{1}-\frac{3ip}{\beta_{2p}}\mathbf{e}_{2p}.
\end{equation}
For the right-hand side of (\ref{funda}) we directly calculate 
\begin{equation} \label{lh}
2^{-a}ip(b+1)\mathbf{e}_{2p}.
\end{equation}
Comparing (\ref{rh}) and (\ref{lh}) implies $p\gamma_p=0$, which is a contradiction to $p\ne 0$ and $\gamma_p\ne 0$. This completes the argument if $a\in\R\setminus\{1\}$.
\\

In the remaining case of $a=1$ the relation (\ref{kp}) is void. Here we may only use (\ref{sign}) to conclude that
\begin{equation*}
b=k\,\hbox{sign}\, p<-\vert p\vert\le -1.
\end{equation*}
This completes the proof, since we assumed that $b\ge -1$ if $a =1$.

\end{proof}

\begin{cor}
[$\mu$-DP]
The generalized Degasperis-Procesi flow, described by the initial value problem 
\begin{eqnarray*} 
\begin{cases}
m_{t} + u m_{x} + 3 u_x m = 0,\quad \; m :=-\partial_x^2 u+\int_{\s^1}u\, dx,\\[0.1cm]
u(0,\cdot) = u_0, \quad \quad \quad \quad u_0 \in  \CS,
\end{cases}
\end{eqnarray*}
cannot be realized as an Euler equation for any regular inertia operator $A \in \mathcal L^{sym}_{is}(\CS)$. 
\end{cor}

In applications, one often aims at normalizing 
solutions of the flow (\ref{proto}). One way to do so is to consider functions with zero spatial mean. More precisely, let 

$$\hCS:=\{u\in \CS\,;\,\int_{\s^1} u(x)\,dx=0\}.$$
Observe that $\Lambda_\mu^a\vert\hCS=(-\Delta)^{a/2}$ and that\footnote{Here and in the following we use the notation $(-\Delta)^0:=\hbox{id}_{\s^1}$.} 
$$
(-\Delta)^{a/2}\in \mathcal L^{sym}_{is}(\hCS)\quad\hbox{for all}\quad a\in \R.
$$
On the other hand, 
the homogeneous space {\sc Diff}$^\infty(\s^1)/\hbox{Rot}(\s^1)$, i.e., the coset manifold of {\sc Diff}$^\infty(\s^1)$ modulo the subgroup of rigid rotations, can naturally be identified with the Fr\'{e}chet Lie group 
$$
\hbox{{\sc Diff}}_0^\infty(\s^1):=\{\varphi\in \hbox{\sc{Diff}}^\infty(\s^1)\,;\, \varphi(x_0)=x_0\},$$ 
where $x_0\in \s^1$ is fixed henceforth. Indeed, the mapping
$$
\hbox{{\sc Diff}}_0^\infty(\s^1)\to \hbox{{\sc Diff}}^\infty(\s^1)/\hbox{Rot}(\s^1),\quad \varphi\mapsto [\varphi]
$$
is a smooth diffeomorphism. Moreover, in a sufficiently small neighbourhood $U$ of $\hbox{id}_{\s^1}$ in 
$\hbox{{\sc Diff}}_0^\infty(\s^1)$, each $\varphi\in U$ may be written as $\varphi = \hbox{id}_{\s^1} + u$ with some $u\in \CS$ such that $u(x_0)=0$. Hence the Lie algebra of $\hbox{{\sc Diff}}_0^\infty(\s^1)$ is represented by
$$
\nCS:=\{u\in\CS\,;\, u(x_0)=0\},
$$ 
which is canonically isomorphic to $\hCS$. Thus if we restrict (\ref{proto}) to $\hbox{{\sc Diff}}_0^\infty(\s^1)$, we get
\begin{eqnarray} \label{protoR0}
\begin{cases}
m_t + u m_x + b \; u_x m = 0, \quad t > 0,\ x\in\s^1,\\[0.1cm]
m = (-\Delta)^{a/2} u. \\[0.1cm]
\end{cases}
\end{eqnarray}

Note, however, that evolutions on the full Lie group $\hbox{{\sc Diff}}^\infty(\s^1)$
cannot be controlled by flows on $\hbox{{\sc Diff}}_0^\infty(\s^1)$. Thus Proposition \ref{MR} is not suitable to derive restrictions on the geometry of (\ref{protoR0}). In addition, it is worthwhile to mention that the proof of Proposition \ref{MR}
is crucially based on scaling $u\mapsto u+\lambda \mathbf{1}$, which is obviously useless for (\ref{protoR0})
on the Lie algebra $\hCS$.

We study (\ref{protoR0}) in the particular case,  when the regular inertia operator $A_{a,b}$ is in addition a Fourier multiplication operator. Such operators will be called {\em{regular inertia operators of Fourier type}}.

Given $a\in\R$, we define 
\begin{equation*}
E_a:=\left\{-\frac{2^{a+1}+1}{2^a+2},\,-\frac{3^{a+1}+1}{3^a+3}\right\}\quad\hbox{and}\quad R_a:=Q_a^{-1}(0)\cap\R,
\end{equation*}
where $Q_a$ is the quadratic polynomial defined in (\ref{Q}). We remark that $R_a\ne\emptyset$ for all $a\in\R$, cf. (\ref{Q}).
\\

Then we have the following result.
 
\begin{prop}\label{meanF}
Let $a,\,b\in\R$ be given and consider the doubly parameterized family of fractional differential equations
\begin{eqnarray} \label{protoR}
\begin{cases}
m_t + u m_x + b \; u_x m = 0, \quad t > 0,\ x\in\s^1,\\[0.1cm]
m = (-\Delta)^{a/2} u \\[0.1cm]
\end{cases}
\end{eqnarray}
on {\sc Diff}$_0^\infty(\s^1)$. Assume  further that there is an
inertia operator ${A_{a,b}}\in \mathcal L^{sym}_{is}(\hCS)$ of Fourier type such that
(\ref{protoR}) 
is the Euler equation with respect to the metric $\rho_{A_{a,b}}$.
If $b\not\in E_a\cup R_a$, then $b = 2$ and $A_{a,2} = (-\Delta)^{a/2}$.
\end{prop}

\begin{proof}
(a) Let $A:={A_{a,b}}\in \mathcal L^{sym}_{is}(\hCS)$ be of Fourier type and write $(\beta_k)_{k\in\mathbb{Z}}\in \mathbb{C}^{\mathbb{Z}}$ for the symbol of $A$.
Then $\beta_0=0$ and $\beta_k\ne 0$ for all $k\in\mathbb{Z}^\ast$.  Without restriction we may assume that $\beta_1=1$.
To simplify our notation we set $\Lambda^a:=(-\Delta)^{a/2}$ for $a\in\R$.
Furthermore note that $A$ and $\Lambda^a$ commute. Thus (\ref{proto1}), 
Proposition \ref{gEuler}, and (\ref{chris}) imply that
\begin{equation} \label{fundaF}
\Lambda^{a}[2 Au \;u' + u \;(Au)'] = A [b \Lambda^a_\mu u \;u' + u \;(\Lambda^a_\mu u)']
\end{equation}
for all $u\in \hCS$.

(b) Consider first the case $b=-1$. Recalling that $\mathbf{e}_k(x):=e^{ikx}$ for $k\in\mathbb{Z}^\ast$, substitution of 
$u=\mathbf{e}_k$ in (\ref{fundaF})  implies that $A\equiv 0$, which is not possible. Thus we may
assume that $b\ne -1$.

(c) Next we choose $u=\mathbf{e}_k+\mathbf{e}_{-k}$ in (\ref{fundaF}) and find from the coefficient of $\mathbf{e}_{0}$
that
\begin{equation}\label{symm}
\beta_k=\beta_{-k}\quad\hbox{for all}\quad k\in\mathbb{Z}.
\end{equation}
Thus it suffices to consider the case $k\ge 1$ and the to prove that $\beta_k=k^a$ for $k\ge 1$.

(d) Setting now $u=\mathbf{e}_k$ in (\ref{fundaF}), we get
\begin{equation}\label{2k}
\beta_{2k}=\frac{\,3\cdot 2^a\,}{b+1}\,\cdot\,\beta_k\quad\hbox{for all}\quad k\in\mathbb{Z}^\ast.
\end{equation}
On the other hand, setting $u=\mathbf{e}_k+\mathbf{e}_1$ for $k\ge 1$ in (\ref{fundaF}), we get from the corresponding coefficient of $\mathbf{e}_{k+1}$ the recursion formula
\begin{equation}\label{rekursion}
\beta_{k+1}=(k+1)^a\frac{\beta_k(2+k)+2k+1}{\,b(k^a+k)+k^{a+1}+1\,}\quad\hbox{for all}\quad k\ge 1
\end{equation}
with $b(k^a+k)+k^{a+1}+1\ne 0$. Note that our assumption $b\not\in E_a$ precisely ensures that $\beta_3$ and $\beta_4$ are well-defined. In fact, we shall see later on that each of the coefficients $\beta_k$ is well-defined.
With $\beta_1=1$ we get from (\ref{rekursion}) that
\begin{equation}\label{b2}
\beta_2= \frac{2^a\cdot 3}{b+1}
\end{equation}
and 
\begin{equation}\label{b4}
\beta_4= 4^a\,\frac{7(b+1)(2^a(b+2)+2b+1)+3^a\cdot 25(b+1)+60\cdot 6^a}{(b+1)(2^a(b+2)+2b+1)(3^a(b+3)+3b+1)}\,.
\end{equation}
Inserting further (\ref{b2}) into (\ref{2k}) we also find
\begin{equation}\label{b4'}
\beta_4=4^a\,\frac{9}{(b+1)^2}.
\end{equation}
Equalising (\ref{b4}) and (\ref{b4'}) leads to a cubic polynomial $P(b)$ in $b$, whose real zeros determine the admissible values of $b$ in (\ref{fundaF}). Elementary calculations yield
\begin{equation}
P(b)=a_3(b-2)^3+ a_2(b-2)^2 +a_1 (b-2),
\end{equation}
where the coefficients are given by:
\begin{eqnarray*}
a_3 &=& 7 (2^a + 2) \\
a_2 &=& 43\cdot 2^a + 7\cdot3^a - 9\cdot6^a + 65 \\
a_1 &=& 60\cdot2^a + 15\cdot3^a - 21\cdot6^a + 75.
\end{eqnarray*}
Thus $b=2$ is a zero of $P$. Moreover, one checks that there are two other real zeros of $P$. Indeed, setting
\begin{equation}\label{Q}
Q_a(b):= a_3(b-2)^2+a_2(b-2)+ a_1,
\end{equation}
it follows from the explicit expressions of the coefficients $a_3$, $a_2$, $a_1$ that $R_a = Q_a^{-1}(0)\cap\R$ is not empty for any  
$a\in\R$.
But our assumption ensures that $b\not\in R_a$. Hence $b=2$ is the only admissible root of $P$. Inserting $b=2$ into (\ref{rekursion}), an induction argument shows that $\beta_k=k^a$ for all $k\ge 1$. In view of (\ref{symm}), this completes the proof.
\end{proof}

In order to apply Proposition \ref{meanF} to the case $a=1$, let $\mathcal{H}$ denote the Hilbert transform on $\s^1$. Then we have

\begin{cor}
Neither De Gregorio's vorticity model equation
$$
\omega_t + u \omega_x + \omega \h \omega = 0, 
$$
nor the quasi-geostrophic model equation
$$
\omega_t + u \omega_x = \omega \h \omega,
$$
where $\; \omega := {\mathcal {H}}\partial u=(-\Delta)^{1/2}u$ and $u \in \hCS$, can be realized as Euler equations for any regular inertia operator $A \in \mathcal L^{sym}_{is}(\hCS)$ of Fourier type.
\end{cor}

\begin{proof}
Note that 
$$
E_1\cup R_1=\left\{-\frac{5}{3},\,-\frac{5}{4},\,-\frac{5}{7},\,\frac{1}{2}\right\}.
$$
Since here $b=\pm 1$, the result follows from Proposition \ref{meanF}.
\end{proof}

\begin{cor}
The axisymmetric Euler flow in $\R^d$, described by the initial value problem 
\begin{eqnarray*}
\begin{cases}
m_{t} + u m_{x} + \frac{d - 3}{d - 1} \; u_x m = 0,\quad \; m := -\partial_x^2\; u, \\
u(0,.) = u_0, \quad \quad \quad \quad u_0 \in \hat C^\infty(\s^1),
\end{cases}
\end{eqnarray*}
cannot be realized as an Euler equation for any regular inertia operator $A \in \mathcal L^{sym}_{is}(\hCS)$ of Fourier type. 
\end{cor}
\begin{proof}
Elementary calculations show that 
$$\frac{d-3}{d-1}\not\in E_2\cup R_2$$ for all $d\ge 1$.
\end{proof}
Observe that the axisymmetric Euler flow in 2D is also known as the Proudman-Johnson equation, cf. \cite{okamoto}.

We conclude our study by drawing a further consequence of Proposition \ref{meanF} for the flow (\ref{proto1}) in the case of first order inertia operators of Fourier type.

\begin{cor}
Given any number $b \in \R$, consider the family of differential equations
\begin{eqnarray} \label{proto2}
\begin{cases}
m_t + u m_x + b \; u_x m = 0, \quad t > 0,\ x\in\s^1,\\[0.1cm]
m = (-\Delta)^{1/2} u+\int_{\s^1} u 
\end{cases}
\end{eqnarray}
and  assume that there is an
inertia operator ${A_{b}}\in \mathcal L^{sym}_{is}(\CS)$ of Fourier type such that
(\ref{proto2}) 
is the Euler equation on {\sc Diff}$^\infty(\s^1)$ with respect to the metric $\rho_{A_{b}}$ induced by $A_{b}$.
If $b\not\in\{-\frac{5}{3},\,-\frac{5}{4}\}$, then $b = 2$ and $A_{2} = (-\Delta)^{1/2} +\mu$. 
\end{cor}

\section*{Acknowledgments}

We are grateful to Bogdan V. Matioc for carefully checking laborious computations.  \par
The first author expresses his gratitude to Hisashi Okamoto and the Research Institute for Mathematical Sciences at Kyoto University, whose hospitality he highly appreciated during his visit in March 2010. 
\par
The second author acknowledges financial support by the Postdoctoral Fellowship P09024 of the Japan Society for the Promotion of Science.

\end{document}